\documentclass[10pt]{article}
\usepackage[all]{xy}
\usepackage{amsfonts,amsmath,oldgerm,amssymb,amscd}
\newcommand{\ra}{\rightarrow}

\newcommand{\ol}{\overline}

\newtheorem{theorem}{Theorem}[section]
\newtheorem{proposition}[theorem]{Proposition}
\newtheorem{lemma}[theorem]{Lemma}

\newtheorem{corollary}[theorem]{Corollary}

\newcommand{\gj}{\square}

\newcommand{\BQ}{\mbox{$\mathbb Q$}}

	\newcommand{\p}{\mbox{$\mathfrak p$}}

   \newcommand{\lnd}{\mbox{\rm LND}}

\begin{document} 

\begin{center}
{\Large \bf A note on rigidity and triangulability of a
        derivation}\\\vspace{.2in} {\large 
        Manoj K. Keshari and Swapnil A. Lokhande}\\
\vspace{.1in}
{\small
Department of Mathematics, IIT Bombay, Mumbai - 400076, India;\;
        (keshari,swapnil)@math.iitb.ac.in}
\end{center}

{\small
\noindent{\bf Abstract:} Let $A$ be a $\BQ$-domain, $K=frac(A)$,
$B=A^{[n]}$ and $D\in \lnd_A(B)$. Assume rank $D=$ rank $D_K=r$, where
$D_K$ is the extension of $D$ to $K^{[n]}$. Then we show that

$(i)$ If $D_K$ is rigid, then $D$ is rigid.

$(ii)$ Assume $n=3$, $r=2$ and $B=A[X,Y,Z]$ with $DX=0$. Then $D$ is
triangulable over $A$ if and only if $D$ is triangulable over
$A[X]$. In case $A$ is a field, this result is due to Daigle.

\vspace*{.1in}
\noindent {\bf Mathematics Subject Classification (2000):}
Primary 14L30, Secondary 13B25. 
\vspace*{.1in}

\noindent {\bf Key words:} Locally nilpotent derivation, rigidity,
triangulability.  }

\section{Introduction}
Throughout this paper, {\bf $k$ is a field and all rings are
$\BQ$-domains.}  We will begin by setting up some notations from \cite{7}. 
Let $B=A^{[n]}$ be an $A$-algebra, i.e.  $B$ is
$A$-isomorphic to the polynomial ring in $n$ variables over $A$.  A
{\it coordinate system} of $B$ over $A$ is an ordered $n$-tuple $(X_1,
X_2,..., X_n )$ of elements of $B$ such that $A[X_1, X_2,..., X_n]=B$.

An $A$-derivation $D:B\ra B$ is {\it locally nilpotent} if for each
$x\in B$, there exists an integer $s>0$ such that $D^s(x)=0$; $D$ is
{\it triangulable} over $A$ if there exists a coordinate system
$(X_1,\ldots,X_n)$ of $B$ over $A$ such that 
$D(X_i)\in A[X_1,\ldots,X_{i-1}]$ for $1\leq i\leq n$; {\it rank} of
$D$ is the least integer $r\geq 0$ for which there exists a coordinate
system $(X_1,\ldots,X_n)$ of $B$ over $A$ satisfying
$A[X_1,\ldots,X_{n-r}]\subset$ ker $D$; 
$\lnd_A(B)$ is the set of all locally
nilpotent $A$-derivations of $B$.

Let $\Gamma(B)$ be the set of coordinate systems of $B$ over
$A$. Given $D\in \lnd_A(B)$ of rank $r$, let $\Gamma_D(B)$ be the set
of $(X_1,\ldots,X_n)\in \Gamma(B)$ satisfying
$A[X_1,\ldots,X_{n-r}]\subset$ ker $D$; $D$ is {\it rigid} if
$A[X_1,..., X_{n-r}]=A[X'_1,..., X'_{n-r}]$ holds whenever $(X_1,...,
X_{n})$ and $(X'_1,..., X'_{n})$ belong to $\Gamma_D(B)$. 

For an example, if $D\in \lnd_A(B)$ has rank $1$, then $D$ is
rigid. In this case $ker(D)=A[X_1,\ldots,X_{n-1}]$ for some coordinate
system $(X_1,\ldots,X_n)$ and $D=f\partial_{X_n}$ for some $f\in
ker(D)$. If rank $D=n$, then $D$ is obviously rigid, as no variable is
in $ker(D)$. If rank $D\not= 1,n$, then $ker(D)$ is not generated by
$n-1$ elements of a coordinate system and is generally difficult to
see whether $D$ is rigid.  For an example of a non-rigid triangular
derivation on $k^{[4]}$, see section 3.  We remark that there is also
a notion of a ring to be rigid. We say that a ring $A$ is rigid if
$\lnd(A)=\{0\}$, i.e. there is no non-zero locally nilpotent
derivation on $A$. Clearly polynomial rings $k^{[n]}$ are non-rigid
rings for $n\geq 1$.

We will state the following result of Daigle (\cite{7}, Theorem 2.5)
which is used later.

\begin{theorem}\label{daigle}
All locally nilpotent derivations of $k^{[3]}$ are rigid.
\end{theorem}

Our first result extends this as follows:

\begin{theorem}\label{Daigle1}
Let $A$ be a ring, $B=A^{[n]}$, $K=frac(A)$ and $D\in
\lnd_A(B)$. Assume that rank $D=$ rank $D_K$, where $D_K$ is the
extension of $D$ to $K^{[n]}$. If $D_K$ is rigid, then $D$ is rigid.
\end{theorem}

In (\cite{7}, Corollary 3.4), Daigle obtained the following
triangulability criteria: Let $D$ be an irreducible, locally nilpotent
derivation of $R=k^{[3]}$ of rank at most 2. Let $(X, Y, Z)\in
\Gamma(R)$ be such that $DX=0$. Then $D$ is triangulable over $k$ if
and only if $D$ is triangulable over $k[X]$. Our second result extends
this result as follows:

\begin{theorem}\label{Daigle2}
Let $A$ be a ring, $B=A^{[3]}$, $K=\mbox{frac}(A)$ and
 $D\in \lnd_A(B)$. Let $(X, Y, Z)\in \Gamma(B)$ be such that
 $DX=0$. Assume that rank $D=$ rank $D_K=2$.  Then $D$ is triangulable
 over $A$ if and only if $D$ is triangulable over $A[X]$.
\end{theorem}

\section{Preliminaries}

Recall that a ring is called a {\it HCF}-ring if intersection of two
 principal ideal is again a principal ideal. We state some results for
 later use.

\begin{lemma}(Daigle \cite{7}, 1.2){\label{l1}}
Let $D$ be a $k$-derivation of $R=k^{[n]}$ of rank 1 and let $(X_1,
X_2,..., X_n)\in \Gamma(R)$ be such that $k[X_1,
X_2,..., X_{n-1}]\subset$ ker $D$. Then

$(i)$ ker $D=k[X_1, X_2,..., X_{n-1}];$

$(ii)$ $D$ is locally nilpotent if and only if $D(X_n)\in$ ker $D.$
\end{lemma}

\begin{proposition}(Abhyankar, Eakin and Heinzer \cite{9}, Proposition 4.8){\label{l2}}
Let $R$ be a HCF-ring, $A$ a ring of transcendence degree one over $R$
and $R\subset A\subset R^{[n]}$ for some $n\geq 1.$ If $A$ is a
factorially closed subring of $R^{[n]},$ then $A=R^{[1]}.$
\end{proposition}

\begin{lemma}(Abhyankar, Eakin and Heinzer \cite{9}, 1.7){\label{l10}}
Suppose $A^{[n]}=R=B^{[n]}$. If $b\in B$ is such that $bR \cap A \neq
0$, then $b\in A.$
\end{lemma}

\begin{theorem}(\cite{10}, Theorem 4.11){\label{l3}}
 Let $R$ be a HCF-ring and $0\neq D\in \lnd_R(R[X,Y])$. Then there
 exists $P\in R[X, Y]$ such that ker $D=R[P].$
\end{theorem}

\begin{theorem}{\label{BD5}}
(Bhatwadekar and Dutta \cite{50})
Let $A$ be a ring and $B=A^{[2]}$. Then $b\in B$ is a variable of $B$
over $A$ if and only if for every prime ideal $\p$ of $A$, $\ol b \in
\ol B:=B_{\p}/\p B_{\p}$ is a variable of $\ol B$ over $A_{\p}/\p
A_{\p}$.
\end{theorem}


\section{Rigidity} 

\begin{theorem}\label{Daigle3}
Let $A$ be a ring, $B=A^{[n]}$, $K=frac(A)$ and $D\in
\lnd_A(B)$. Assume that rank $D=$ rank $D_K$, where $D_K$ is the
extension of $D$ to $K^{[n]}$. If $D_K$ is rigid, then $D$ is rigid.
\end{theorem}

\begin{proof}
Assume rank $D=$ rank
$D_K=r$ and $D_K$ is rigid. We need to show that $D$ is rigid, i.e.
if $(x_1,\ldots,x_n)$ and $(y_1,\ldots,y_n)$ are two coordinate
systems of $B $ satisfying $A[x_1,\ldots, x_{n-r}]\subset$ ker $D$ and
$ A[y_1,\ldots, y_{n-r}]\subset$ ker $(D)$, then we have to show that
$A[x_1,\ldots, x_{n-r}]=A[y_1,\ldots, y_{n-r}].$ By symmetry, it is
enough to show that $A[x_1,\ldots, x_{n-r}]\subset A[y_1,\ldots,
y_{n-r}]$.

Since $D_K$ is rigid and rank $D_K=r$, we get $K[x_1,\ldots,
x_{n-r}]=K[y_1,\ldots, y_{n-r}]$.  If $f\in A[x_1,\ldots, x_{n-r}]$,
then $f\in K[y_1,\ldots, y_{n-r}]$. We can choose $a\in A$ such that
$af\in A[y_1,\ldots, y_{n-r}]$ and hence $fB\cap A[y_1,\ldots,
y_{n-r}]\neq 0$. Applying (\ref{l10}) to $A[x_1,\ldots,
x_{n-r}]^{[r]}=B=A[y_1,\ldots, y_{n-r}]^{[r]}$, we get $f\in
A[y_1, y_2,\ldots, y_{n-r}].$ Therefore $A[x_1,\ldots, x_{n-r}]\subset
A[y_1,\ldots, y_{n-r}]$. This completes the proof. $\hfill \gj$
\end{proof}

The following result is immediate from (\ref{Daigle3}) and (\ref{daigle}).

\begin{corollary}{\label{l12}}
Let A be a ring, $B=A^{[3]},$ $D\in \mbox{LND}_A(B).$ If rank $D$=
rank $D_K$, then $D$ is rigid.
\end{corollary}

\begin{remark}
$(1)$ If $D\in \lnd_A(B)$, then rank $D$ and rank $D_K$ need not be
same.  For an example, consider $A=\mathbb{Q}[X]$ and $B=A[T,Y,Z].$
Define $D\in \lnd_A(B)$ as $DT=0$, $D(Y)=X$ and $D(Z)=Y$. Then rank
$D=2$ and rank $D_K=1.$ Further, $(T'=T-Y^2+2XZ,Y,Z)\in \Gamma_D(B)$
and $A[T]\not= A[T']$. Therefore, $D$ is not rigid, whereas $D_K$ is
rigid, by (\ref{daigle}). 

Above example gives a $D\in \lnd(k^{[4]})$ which is not rigid. Hence
Daigle's result (\ref{daigle}) is best possible. Note that $D$ is a
triangular derivation and by \cite{BD2010}, $ker(D)$ is a finitely
generated $k$-algebra.

$(2)$ The condition in (\ref{Daigle3}) is sufficient but not
necessary, i.e. $D\in \lnd_A(B)$ may be rigid even if rank $D\not=$
rank $D_K$. For an example consider $A=\mathbb{Q}[X]$ and $B=A[Y,Z].$
Define $D\in \lnd_A(B)$ as $D(Y)=X$ and $D(Z)=Y$. Then rank $D=2$ and
hence $D$ is rigid. Further, rank $D_K=1$ and $D_K$ is also rigid, by
(\ref{daigle}).

$(3)$ It will be interesting to know if $D\in \lnd(k^{[n]})$ being
rigid implies that $ker(D)$ is a finitely generated $k$-algebra. The
following example could provide an answer.

Let $D=X^{3} \partial_S+S\partial_T+T\partial_U+X^2\partial_V\in
\lnd(B)$, where $B=k^{[5]}=k[X,S,T,U,V]$. Daigle and Freudenberg
\cite{DF99} have shown that $ker(D)$ is not a finitely generated
$k$-algebra. We do not know if $D$ is rigid.  We will show that rank $D=3$.

Clearly $X,S-XV\in ker(D)$ is part of a coordinate system. Hence rank
$D\leq 3$.
If rank $D=1$, then there exists a coordinate system
$(X_1,\ldots,X_4,Y)$ of $B$ over $k$ such that $X_1,\ldots,X_4\in
ker(D)$. Hence $D=f\partial_Y$ for some $f\in k[X_1,\ldots,X_4]$ and
$ker(D) =k[X_1,\ldots,X_4]$ is a finitely generated $k$-algebra, a
contradiction.
If rank $D=2$, then there exists a coordinate system
$(X_1,X_2,X_3,Y,Z)$ of $B$ over $k$ such that $X_1,X_2,X_3\in
ker(D)$. If we write $A=k[X_1,X_2,X_3]$, then $D\in
\lnd_A(A[Y,Z])$. Since $A$ is UFD, by (\cite{10}, Theorem 4.11),
$ker(D)=A^{[1]}$, hence $ker(D)$ is a finitely generated $k$-algebra, a
contradiction. Therefore, rank of $D$ is $3$.

\end{remark}


\section{Triangulability}

We begin with the following result which is of independent interest.

\begin{lemma}{\label{l13}}
 Let $A$ be a UFD, $K=\mbox{frac}(A)$, $B=A^{[n]}$ and $D\in
 \lnd_A(B)$. Let $D_K$ be the extension of $D$ on $K^{[n]}.$ If $D$ is
 irreducible, then $D_K$ is irreducible.
\end{lemma}

\begin{proof}
 We prove that if $D_K$ is reducible, then so is $D.$ Let $
 D_K(K^{[n]}) \subset fK^{[n]}$ for some $f\in B$. If
 $B=A[x_1,\ldots,x_n]$, then we can write $Dx_i=fg_i/c_i$ for some
 $g_i\in B$ and $c_i\in A$ with gcd$_B(g_i, c_i)=1.$ Since $Dx_i\in
 B$, we get $c_i$ divides $f$ in $B$. If $c$ is 
 lcm of $c_i$'s, then $c$ divides $ f$. If we take $f'=f/c\in B$, then $Dx_i\in
 f'B$ and hence $D$ is reducible. $\hfill \gj$
\end{proof}

\begin{proposition}{\label{l16}}
Let $A$ be a ring, $B=A^{[3]}$, and $D\in \lnd_A(B)$ be
of rank one. Let $(X, Y, Z)\in \Gamma(B)$ be such that
$DX=0$. Assume that either $A$ is a UFD or $D$ is irreducible. Then
$D$ is triangulable over $A[X]$.
\end{proposition}

\begin{proof}
 As rank $D=1$, there exists $(X',Y', Z')\in \Gamma(B)$ such that
 $DX'=DY'=0$. By (\ref{l1}), ker $D=A[X', Y']$ and $DZ'\in$ ker $D$.

$(i)$ Assume $A$ is a UFD. Since $A[X]\subset A[X', Y']\subset
A[X]^{[2]}$ and $A[X', Y']$ is factorially closed in $A[X]^{[2]}$; by
(\ref{l2}), $A[X', Y']=A[X][P]$ for some $P\in B.$ Hence $B=A[X, P,
Z']$ and $DZ'\in A[X, P]$. Thus $D$ is triangulable over $A[X].$

$(ii)$ Assume $D$ is irreducible.  Then $DZ'$ must be a unit.  To show
that $X$ is a variable of $A[X',Y']$ over $A$.  By (\ref{BD5}), it is
enough to prove that for every prime ideal $\p$ of $A$, if
$\kappa(\p)=A_{\p}/\p A_{\p}$ then $\ol X$ is a variable of $\kappa(\p)[X',Y']$
over $\kappa(\p)$. Extend $D$ on $A_{\p}[X, Y, Z]$ and let $\ol D$ be
$D$ modulo $\p A_{\p}$.  Then $\mbox{ker }\ol D=\kappa(\p)[X', Y']$.
By (\ref{l2}), $\mbox{ker }\ol D=\kappa(\p)[X]^{[1]}$. Therefore $X$ is
a variable of $A[X', Y']$, i.e.
$A[X', Y']=A[X, P]$ for some $P\in B.$ Hence $B=A[X, P, Z']$. Thus $D$
is triangulable over $A[X].$ $\hfill \gj$
\end{proof}

\begin{proposition}{\label{l17}}
 Let $A$ be a ring, $K=\mbox{frac}(A)$, $B=A^{[3]}$ and
 $D\in \lnd_A(B)$.  Let $(X, Y, Z)\in \Gamma(B)$ be such that
 $DX=0$. Assume rank $D$=rank $D_K=2$. Then $D$ is triangulable over
 $A$ if and only if $D$ is triangulable over $A[X]$.
 \end{proposition}

\begin{proof}
We need to show only $(\Rightarrow)$.
Suppose that $D$ is triangulable over $A$. Then there
  exists $(X',Y', Z')\in \Gamma(B)$ such that $DX'\in A,$ $ DY'\in
  A[X']$ and $DZ'\in A[X', Y']$. 
If $a=DX'\neq 0$, then $D_K(X'/a)=1$;
which implies that rank $D_K=1,$ a
  contradiction. Hence $DX'=0$.

Since $D_K$ is rigid, by
  (\ref{Daigle3}), $D$ is rigid of rank $2$. Therefore $A[X]=A[X']$
  and $D$ is triangulable over $A[X].$ $\hfill \gj$
\end{proof}

{\bf Acknowledgements.} We sincerely thank the referee for his/her
remarks which improved the exposition.

{}

\end{document}